\documentclass[12pt]{amsart}
\usepackage[T1]{fontenc}
\usepackage{latexsym} 
\usepackage{amsfonts, amsthm, amsmath,amssymb}
\usepackage{BOONDOX-calo}
\usepackage{geometry}
\usepackage{times}
\usepackage{newtxtext,newtxmath}
\usepackage{xcolor}
\definecolor{refkey}{rgb}{0,0,1}
\definecolor{labelkey}{rgb}{1,0,0}
\usepackage{graphicx}
\usepackage{float}
\usepackage{mdframed}
\usepackage{ulem}
\definecolor{darkblue}{rgb}{0.0, 0.0, 0.55}
\definecolor{darkcerulean}{rgb}{0.03, 0.27, 0.49}
\definecolor{darkpowderblue}{rgb}{0.0, 0.2, 0.8}
\definecolor{britishracinggreen}{rgb}{0.0, 0.26, 0.15}
\newenvironment{blu}{\color{darkpowderblue}}{}
\newenvironment{mgg}{}{}
\newenvironment{brg}{\color{britishracinggreen}}{}
\newenvironment{red}{\color{red}}{}
\newenvironment{grn}{\color{green}}{}
\newcommand{\bgr}{\begin{grn}}
	\newcommand{\egr}{\end{grn}}
\newcommand{\bre}{\begin{red}}
	\newcommand{\ere}{\end{red}}
\newcommand{\bas}{\begin{brg}}
	\newcommand{\eas}{\end{brg}}
\newcommand{\bblu}{\begin{blu}}
	\newcommand{\eblu}{\end{blu}}
\newcommand{\bmag}{\begin{mgg}}
	\newcommand{\emag}{\end{mgg}}

\NewDocumentCommand{\colorrule}{O{.4pt}m}{{\color{#2}\hrule height#1}\vspace{4mm}}
\newtheorem{thm}{Theorem}[section]

\newtheorem{lem}[thm]{Lemma}
\newtheorem{coro}[thm]{Corollary}
\theoremstyle{definition}
\newtheorem{defn}[thm]{Definition}
\newtheorem{rem}[thm]{Remark}

\numberwithin{equation}{section}

\def\cH{{\mathcal H}} 
\def\cA{{\mathcal A}}

\def\beq{\begin{equation}} 
	\def\eeq{\end{equation}} 
 
\def\ot{{\otimes}} 
\def\wres{\mathcal{W}}

\def\Tr{\hbox{Tr}}

\def\IC{\mathbb{C}}

\def\IZ{\mathbb{Z}}

\newcommand{\ket}[1]{|#1\rangle}    
\newcommand{\up}{{\mathord{\uparrow}}} 
\newcommand{\dn}{{\mathord{\downarrow}}} 
\newcommand{\shalf}{{\scriptstyle\frac{1}{2}}} 
\newcommand{\oh}{{\tfrac{1}{2}}}    
\newcommand{\half}{{\mathchoice{\oh}{\oh}{\shalf}{\shalf}}} 
\newcommand{\ssesq}{{\scriptstyle\frac{3}{2}}} 
\newcommand{\sesq}{{\mathchoice{\ooh}{\ooh}{\ssesq}{\ssesq}}} 
\newcommand{\ooh}{{\tfrac{3}{2}}}   
\newbox\ncintdbox \newbox\ncinttbox
\setbox0=\hbox{$-$}
\setbox2=\hbox{$\displaystyle\int$}
\setbox\ncintdbox=\hbox{\rlap{\hbox
		to \wd2{\kern-.1em\box2\relax\hfil}}\box0\kern.1em}
\setbox0=\hbox{$\vcenter{\hrule width 4pt}$}
\setbox2=\hbox{$\textstyle\int$}
\setbox\ncinttbox=\hbox{\rlap{\hbox
		to \wd2{\kern-.05em\box2\relax\hfil}}\box0\kern.1em}
\newcommand{\ncint}{\mathop{\mathchoice{\copy\ncintdbox}%
		{\copy\ncinttbox}{\copy\ncinttbox}%
		{\copy\ncinttbox}}\nolimits}
\DeclareMathOperator{\Res}{Res}     
\newcommand{\ox}{\otimes}           
\title{Spectral Torsion} 
\author[L.\ D\k{a}browski]{Ludwik D\k{a}browski${}^{(1)}$}
\address{${}^{(1)}$ SISSA (Scuola Internazionale Superiore di Studi Avanzati), \newline\indent Via Bonomea 265, 34136 Trieste, Italy.} 
\email{dabrow@sissa.it} 
\author[A.\ Sitarz]{Andrzej Sitarz${}^{(2)}$}
\author[P.\ Zalecki]{Pawe\l{} Zalecki${}^{(2)}$}
\thanks{This work is supported by the Polish National Science Centre grant 2020/37/B/ST1/01540}
\address{${}^{(2)}$ Institute of Theoretical Physics, Jagiellonian University, \newline\indent
	prof.\ Stanis\l awa \L ojasiewicza 11, 30-348 Krak\'ow, Poland.}
\email{andrzej.sitarz@uj.edu.pl}  
\email{pawel.zalecki@doctoral.uj.edu.pl}
\date{}
\begin{document}
\maketitle
\begin{abstract}
We introduce a trilinear functional of differential one-forms for a finitely summable regular spectral triple with a noncommutative residue.
We demonstrate that for a canonical spectral triple over a closed spin manifold it recovers the torsion
of the linear connection. We examine several spectral triples, including Hodge-de\,Rham, Einstein-Yang-Mills, almost-commutative two-sheeted space, 
conformally rescaled noncommutative tori, and quantum $SU(2)$ group, showing that the third one has a nonvanishing torsion if nontrivially coupled.
\end{abstract}
\vspace{3mm}
\section{Introduction}
The existence and uniqueness of a metric-compatible linear connection with vanishing torsion is one of the fundamental theorems of Riemannian geometry. Torsion appears quite naturally as a vector-valued two-form in this approach, and the assumption that it identically vanishes leads to the Levi-Civita connection, which solely depends on the metric. In the background of general relativity lies the torsion-free condition of Riemannian geometry, which very accurately describes the gravitational interaction of bodies.
Torsion, on the other hand, has a physical interpretation as the quantity that measures the twisting of reference frames along geodesics and has been considered an independent field in physics in Einstein-Cartan theory \cite{Ca23}. Torsion, unlike gravitational fields, must vanish in a vacuum and does not propagate; however, its existence causes nonlinear interactions of matter with spin and has the potential to change the standard singularity theorems of General Relativity (see \cite{HHKN76,Sh02}
 for a review of physical theories).

The emergence of noncommutative geometry \cite{Co80,Co94}, which generalises standard notions of differential geometry to an algebraic (or operator-algebraic) setup, has raised new questions about the concepts of linear connections, metric, and torsion.  A transparent link between different approaches, ranging from purely algebraic noncommutative differential geometry to operator algebraic spectral triple formalism, and a unifying view of metric, linear connection, and torsion has yet to be established. Although the algebraic concepts of linear connections, torsion, and metric compatibility have been achieved, the existence of the Levi-Civita connection can only be proven in special cases (see \cite{BM20, BGJ21a,BGJ21b} and the references therein). {
	It is also worth noting an alternative approach, which introduces a notion of curvature  at the operator-algebraic level using Hilbert modules and unbounded Kasparov modules  \cite{MRvS22}.}
	
On the other hand, the spectral triple approach, with the Dirac operator being the fundamental object of geometry, has not yet been able to determine whether the constructed Dirac operators correspond to the Levi-Civita connection or whether they contain a nonvanishing torsion. So far, for that purpose, one could only try to minimise the spectral functional corresponding to the integrated scalar curvature while keeping the metric defined by the Dirac operator unchanged.
Even in the classical case of manifolds, this can be a difficult task
(see computations of spectral action for a nonvanishing torsion \cite{HPS10, PfSt12,ILV08}),  which becomes even harder in a genuinely noncommutative situation \cite{Si14}.  {In either case, even if a noncommutative
counterpart of a Levi-Civita connection exists, the path to define Ricci and
scalar curvature is not unique. In the spectral case, for a chosen Dirac or
Laplace-type operator one can define scalar curvature functional but it is
entangled with the volume element. Similarly, from the Einstein functional
as defined in \cite{DSZ23} it is not fully clear whether one can separate from
it the counterpart of the Ricci tensor alone. }
	
The aim of this paper is to propose a plain, purely spectral method that allows to determine the torsion as the density of the torsion functional and impose (if possible) the torsion-free condition for regular finitely summable spectral 
triples. We consider it as a first step towards linking the spectral approach
with the algebraic approach based on Levi-Civita connections.
	
	\section{Dirac operator with torsion}
	
	Let us assume that $M$ is a closed spin manifold of dimension $n$, with the metric tensor 
{	
	$g$.
In terms of a (local) basis $e_i$ of orthonormal vector fields on $M$ the 
{metric preserving}
covariant derivative can be written as 
	\begin{equation}\label{connform1}
		\nabla_{e_i}e_j = \omega_{ijk} e_k
	\end{equation}
where the connection coefficients $\omega_{ijk}$  split into the Levi-Civita part $\omega^{LC}_{ijk}$ and the contorsion part $\tau_{ijk}$: 
\begin{equation}\label{connform2}
		\omega_{ijk} = \omega^{LC}_{ijk} + \tau_{ijk}.
	\end{equation}
The first can be expressed explicitly through the structure constants $c_{ijk}$  of the commutators of orthonormal vector fields $e_i$,
	\begin{equation}\label{connform}
		\omega^{LC}_{ijk}:= \frac{1}{2} (c_{ijk}+ c_{kij} + c_{kji}), \qquad [ e_i, e_j ] =: c_{ijk} e_k.
	\end{equation}
	and the remaining part through the torsion of the                                                 
	connection, 
	$$ \tau_{ijk}=\half( T_{ijk} + T_{kij}+ T_{kji} ),
	\qquad T(e_i, e_j):= \nabla_{e_i} e_j-\nabla_{e_j} e_i - [e_i,e_j] =: T_{ijk} e_k.$$
	Note the antisymmetry 
$\omega^{LC}_{ijk}=-\omega^{LC}_{ikj}$, $\tau_{ijk}=-\tau_{ikj}$ while $c_{ijk}= - c_{jik}$, 
{	$T_{ijk}= - T_{jik}$,
	}	
and that one can re-express
	$$ c_{ijk}= \omega^{LC}_{ijk}- \omega^{LC}_{jik} {\rm \ and \ } T_{ijk}= \tau_{ijk}- \tau_{jik}. $$
	
	The lift of the covariant derivative to Dirac spinor fields reads locally
	\begin{equation}
		\nabla_{e_i} = {\mathcal{L}_{e_i}} - \frac{1}{4}\omega_{ijk}\gamma^j\gamma^k,
	\end{equation}
	where ${\mathcal{L}}$ is the Lie derivative and $\gamma^j$ are the usual generators of Clifford algebra.
	The Dirac operator on a spin manifold is, in a local basis of orthonormal frames,  
	a first-order differential operator,           
	\begin{equation}
{		D_{\tau} 
}
= i \gamma^j \nabla^{}_{e_j}, \label{Dirac}
	\end{equation}
	and is unique for a given metric and contorsion tensor. 
	
	The canonical Dirac operator, 
	$D$, is the one for vanishing  contorsion and we call it torsion-free Dirac operator.	
	}
	The torsion tensor can be split into three parts, the vector part, totally antisymmetric
	part and the Cartan part, however the  Cartan part does not contribute at all to the Dirac
	operator and the vector part has to vanish if the resulting Dirac operator has to be self-adjoint 
	\cite{FrSu78}. Therefore, out of the full torsion  only the antisymmetric part appears in 
	the selfadjoint Dirac operator and hence we shall assume that torsion (and contorsion) arezoom
	antisymmetric tensors. Therefore,  the Dirac operator with a fully antisymetric torsion $T$ is,
	\begin{equation}
		D_T = i \gamma^j \bigl( \nabla^{(LC)}_{e_j} - \frac{1}{8} T_{jkl} \gamma^k \gamma^l \bigr)
		= D - \frac{i}{8} T_{jkl} \gamma^j  \gamma^k \gamma^l, \label{DiracT}
	\end{equation}                    
	where $D$ is the standard Dirac operator, which originates from the Levi-Civita connection.  
We will denote $\hat u$ the Clifford multiplication by the one-form $u$.

	Let $\wres$ denote the Wodzicki residue \cite{Wo87,Gu85}.	We introduce the following definition of the torsion functional 	of which the name is explained by the subsequent torsion recovery theorem.
	\begin{defn}\label{deftfun}
	For $D_T$ given by \ref{DiracT} the trilinear functional of differential one-forms $u,v,w$ 
		\begin{equation}
			{\mathcal T}(u, v,  w) = \wres \bigl( \hat u \hat v \hat w D_T |D_T|^{-n} \bigr), \label{tfun}
		\end{equation}
		is called {\em torsion functional}.
	\end{defn}	
As our first main result we state, 
	\begin{thm}
\label{MThm}
The torsion functional for the Dirac operator $D_T$ on a Riemannian closed spin manifold of
dimension $n$ is, 
\begin{equation}
	{\mathcal T}(u, v,  w) = - 2^{m} i V(S^{n-1})\int\limits_M  \, u_a v_b w_c T_{abc}vol_g.
\label{MThm-fo}
\end{equation}		                                            
where $m=\frac{n}{2}$ for even $n$ and $m=\frac{n+1}{2}$ for odd $n>1$, and $V(S^{n-1})$ is the volume of the unit sphere $S^{n-1}$.
\end{thm}
An immediate consequence is,	
\begin{coro}
The Dirac operator $D_T$ is torsion free iff the torsion functional vanishes:
	$$ T \equiv 0 \;\; \Leftrightarrow \;\; {\mathcal{T}}(\hat u, \hat v, \hat w) = 0. $$
\end{coro}	
To prove the theorem we employ the calculus of pseudodifferential operators and expansion
of the symbols in normal coordinates. 
	
\subsection{Normal coordinates and symbols}
	
Let us recall that in the normal coordinates {(see \cite{SS12}, Chapter 10.3 and \cite{MSV99}) } the metric and the Levi-Civita connection have a Taylor expansion at a given point (${\bf x}=0$) on the manifold,
	\begin{equation}
		\begin{aligned}
			&	g_{ab} = \delta_{ab} - \frac{1}{3} R_{acbd} x^c x^d + o({\bf x^2}),  \\
			&	g^{ab} = \delta_{ab} + \frac{1}{3} R_{acbd} x^c x^d + o({\bf x^2}), \\
			&	\sqrt{\hbox{det}(g)}  = 1  - \frac{1}{6} \mathrm{Ric}_{ab} x^a x^b + o({\bf x^2}), 	\\
			&	\Gamma^a_{bc}(x)= -\frac{1}{3}(R_{abcd}+R_{acbd } ) x^{d} +o({\bf x^2}),			
		\end{aligned}	
	\end{equation}
	where $R_{acbd}$ and $\mathrm{Ric}_{ab}$ are the components of the Riemann and Ricci tensors, respectively, at the point with 
	${\bf x}=0$ and we use the notation $o({\bf x^k})$ to denote that we expand a function up to the polynomial of order $k$ in the normal coordinates.

	As a next step, we compute the symbols of $D_T$ and its inverses as pseudodifferential 
	operators (see Appendix). We start with the symbol of $D_T$ expanded in normal 
	coordinates up to $o({\bf 1})$:
	\begin{equation} 
		\sigma(D_T) = i \gamma^j \biggl( i \xi_j  
		{ - \frac{1}{8}  T_{jps} \gamma^p \gamma^s } \biggr)   + o({\bf 1}).
	\end{equation}
	A Clifford multiplication by a differential form $v$ can be expanded in local normal coordinates 
	around a point on a manifold, as:
	$$ \hat{v} = v_a \gamma^a + o({\bf 1}). $$
	We can now proceed with the proof of the theorem. \\

	\noindent {\sl Proof of Theorem \ref{MThm}.}\ \ 
	First of all, observe that since each one-form is a zero-order differential operator then effectively we need to compute the symbol of order $-n$ of $D_T |D_T|^{-n}$.
	We start with the even dimension case $n=2m$, postponing the case of odd dimensions till later.  To compute the two leading symbols of $D_T^{-2m}$, denoted by 
	$\mathfrak{c}_{2m} + \mathfrak{c}_{2m+1}$ we need the two leading
	symbols of $\sigma(D_T^2) = {\mathfrak a}_2 + {\mathfrak a}_1 + {\mathfrak a}_0 $,
	\begin{equation} \label{D2T}
		\begin{aligned}
			{\mathfrak a}_2 =&  ||\xi||^2  + o({\bf x}), \\
			{\mathfrak a}_1 = & 
			{- \frac{1}{8}  i  T_{abc}\bigl(  \gamma^j \gamma^a \gamma^b \gamma^c   +
				\gamma^a \gamma^b \gamma^c \gamma^j    \bigr)  \xi_j } + o({\bf 1}). 
		\end{aligned}	
	\end{equation}
	Then, using  results from \cite{DSZ23} (Lemma A.1) we have:
	\begin{equation} 
		\begin{aligned}
			&\mathfrak c_{2m} = ||\xi||^{-2m} + o(\bf x),\\
			&\mathfrak c_{2m+1}= 
			{ \frac{1}{8} i m  ||\xi||^{-2m-2} T_{abc}\bigl(  \gamma^j \gamma^a \gamma^b \gamma^c   +
				\gamma^a \gamma^b \gamma^c \gamma^j    \bigr)  \xi_j } 
			+ o({\bf 1}).
		\end{aligned}\label{symD-2m}
	\end{equation}		
	We are now ready to compute the symbol of order $-n$ of $\hat u \hat v \hat w D_T |D_T|^{-n} $:
	
	\begin{equation}
		\begin{aligned} 	 
			\sigma_{-n}(\hat u &\hat v \hat w D_T |D_T|^{-n})  = - \frac{1}{8} 
			u_a \gamma^a  v_b \gamma^b  w_c \gamma^c \biggl(  i  T_{jps} 
			\gamma^j \gamma^p \gamma^s||\xi||^{-2m} \\
			& +  ( - \gamma^k   ) ( - i) m  ||\xi||^{-2m-2} T_{psr}\bigl(  \gamma^j \gamma^p \gamma^s \gamma^r   +
			\gamma^p \gamma^s \gamma^r \gamma^j   
			\bigr) \xi_j \xi_k  \biggr) + o({\bf 1}).
		\end{aligned}
	\end{equation}	
	
	Then, we proceed with the calculation of the Wodzicki residue.
	
	\begin{equation}
		\begin{aligned} 	 
			\wres \bigl( \hat u \hat v \hat w D_T |D_T|^{-n} \bigr) &= 
			\int_M vol_g \, \hbox{Tr}_{Cl}\, \int\limits_{||\xi||=1} \sigma_{-n}(\hat u \hat v \hat w D_T |D_T|^{-n}) \\
			& = - 2^{m} i V(S^{n-1}) \int\limits_M vol_g \, u_a v_b w_c T_{abc},
		\end{aligned} \label{wres-tor}
	\end{equation}	                                                       
 {where we use the standard formulas for the trace of gamma matrices. }
	This ends the proof for even $n$. To see that the proof holds also for odd $n$ we apply
	the usual technique to compute the symbol of $|D_T|=\sqrt{D_T^2}$ following \cite[Chapter XII \S 1, pp 288-299]{Ta81}. 
	It follows that the symbol $|D_T|$ reads, as a sum of homogenous symbols of which we need only the first two (up to relevant order in normal coordinates):
	$$ \sigma(|D_T|) = \biggl(|\xi| + o(\mathbf{x}) \biggl) +
	 \biggl(- \frac{1}{8}  i  T_{abc}\bigl(  \gamma^j \gamma^a \gamma^b \gamma^c   +             
	 \gamma^a \gamma^b \gamma^c \gamma^j    \bigr)  \xi_j |\xi|^{-1} + o({\bf 1})\biggr)
 + \cdots.$$
	Then the rest of the proof is analogous to the case of even dimensions.
	 $\qquad\Box$
	                                       
{
	From \eqref{MThm} it follows that the density of the torsion functional is function-linear and completely antisymmetric in $u,v,w$.
	Therefore it suffices to consider it on three orthonormal 1-forms
	$e^a,e^b,e^c$, with the condition that $a,b,c$ are mutually different 
	numbers from $1$ to $n$. 
	For $n=4$ we compute that $ {\mathcal T}(e^a,e^b,e^c)$ equals 
	(possibly up to a sign) to $W(\gamma^d\gamma^1\gamma^2\gamma^3\gamma^4)$
	where $\{d\}=\{1,2,3,4\} \setminus \{a,b,c\}$.
	For $d=3$  we proceed in a similar way, 
	taking into account, however, that in our irreducible (not faithful) representation $\gamma^1\gamma^2\gamma^3=i$.
	
	Then, using the definition of the Hodge star operator on $M$ and the Clifford relations for gamma matrices, we formulate:
	\begin{rem}
		In low dimensions the torsion functional simplifies
		significantly if we use the Hodge star and take $t = \star (u\wedge v\wedge w)$.
		The lowest odd dimension for which there can be a nonvanishing completely antisymmetric torsion is $n=3$, 
		where 
		$$ {\mathcal T}(u,v,w)  = \wres \bigl(\hat{t} D_T |D_T|^{-n} \bigr), $$
		whereas the lowest possible even dimension is $n=4$, where
		$$ {\mathcal T}(u,v,w)  = - \wres \bigl(\hat{t} \chi D_T |D_T|^{-n} \bigr) $$
		with $\chi =  \gamma^1\gamma^2\gamma^3\gamma^4$.
	\end{rem}
}
	\section{Noncommutative torsion}
	The spectral definition of torsion, Def.\ref{deftfun}, can be readily extended to the noncommutative case of spectral triples. 
	Let  $(\cA, D, \cH)$ be a $n$-summable unital spectral triple, $\Omega^1_D(\cA)$ be the $\cA$ bimodule of one forms generated by  $\cA$ and $[D, \cA]$, which by definition consists of bounded operators on $\cH$. We assume that there exists a generalised 
	algebra of pseudodifferential operators, which contains $\cA$, $D$, $|D|^\ell$ for $\ell\in\IZ$, and there exists a tracial state $\wres$ on it, called a noncommutative residue. Moreover, we assume that the noncommutative residue identically vanishes 
	on $T |D|^{-k}$ for any $k>n$  (c.f. \cite{CoMo95}) and a zero-order operator $T$ (i.e., an operator in the algebra generated by $\cA$ and $\Omega^1_D(\cA)$).
Note that these assumptions are satisfied if the spectral triple is finitely summable and regular, (c.f.  \cite{Hi06} and \cite{Uuye11}).                                                 
	We introduce then the following definition of torsion functional for spectral triples.
	\begin{defn}\label{TorWres}
	{\em Torsion functional} for the spectral triple  $(\cA, D, \cH)$ 	
	is a trilinear functional of  $u,v,w \in \Omega^1_D(\cA)$ defined by,
	$$ \mathcal{T}(u,v,w) = \wres (u v w D|D|^{-n}). $$	
	We say that the spectral triple is torsion-free (or, for simplicity, the Dirac operator $D$ is torsion-free)
	if  $\mathcal{T}$ vanishes identically.
	\end{defn}
	We recall a  stronger condition of  {\em spectrally closed} spectral triples defined in  \cite[Definition 5.5]{DSZ23} which, in particular,  implies the torsion-free condition stated above.
	\begin{defn}
		We say that a spectral triple with a trace on the generalised algebra of pseudodifferential operators is spectrally closed if for any zero-order operator $P$ (from the algebra generated by $\cA,$ and  $\Omega^1_D(\cA)$ )
		the following holds:
		$$ \wres (P D |D|^{-n}) = 0. $$
	\end{defn}
	\begin{lem}[\hbox{\cite[Lemma 5.6]{DSZ23}}]
		The classical spectral triple over a closed spin-c manifold $M$ of dimension $n=2m$ is spectrally closed in the above sense and hence is torsion-free.
	\end{lem}	
	
	\section{Examples}
	
	The  examples of spectral triples presented below are all significant cases of various noncommutative geometries to which we apply the proposed spectral definition of torsion and compute the torsion functional. 
	\subsection{Hodge-Dirac.}
	Let us start with a purely classical (commutative) example of a Hodge-Dirac spectral triple over
	an oriented, closed Riemannian manifold, where the Hilbert space are square-summable
	differential forms and the Dirac is $D=d+\delta$, where $\delta=\ast d \ast$, with
	$\ast$ being the Hodge star operator. In \cite{DSZ23a} we demonstrated that such
	spectral  triple is spectrally closed and therefore torsion-free.

	\subsection{Einstein-Yang-Mills.}
	
Consider a closed spin manifold $M$ of even dimension $n=2m$ and a spectral triple
$$\left(C^\infty(M) \otimes M_N(\mathbb{C}), \widetilde D, L^2(S) \otimes M_N(\mathbb{C})\right).$$ 
Here, the algebra $M_N(\mathbb{C})$ acts by left multiplications on $M_N(\mathbb{C})$ regarded as a Hilbert space with the scalar product $ \langle M,M' \rangle= \hbox{Tr}(M^*M')$ . 
Also, $$\widetilde D = D\otimes \hbox{id}_N + A + JAJ^{-1}$$ 
is a fluctuation of $D\otimes \hbox{id}_N$, where $D$ is the standard Dirac operator on $M$,
$$A\!=\!A^*\!\in \! \,\Omega^1_D = \bigl\{\alpha [D\otimes \hbox{id}_N,\beta] \;|\; \alpha,\beta\in C^\infty(M) \otimes M_N(\mathbb{C}) \bigr\},$$ 
and $J=C\ot *$, 
with $C$ being the charge conjugation on spinors in $S$.
Note that $A$ equals $i\gamma^a X_a$, where $X_a=-X_a^*\in C^\infty(M,M_N(\mathbb{C}))$ acts on 
$M_N(\mathbb{C})$ by left multiplication, while $JAJ^{-1}$ equals   $i\gamma^a$ composed with the right multiplication by $X_a^*$.
Therefore, the fluctuation $A + JAJ^{-1}$ amounts to the action  of $i\gamma^a \mathrm{Ad}_{X_a}$. It follows that the center of this algebra, $C^\infty(M,\IC 1_N)$, acts trivially, 
so we just can consider $X_a$ such that $\langle X_a,1_N \rangle=0$, that is, traceless 
matrix valued functions, $X_a\in C^\infty(M,su(N))$.   
Moreover, we need only the expansion of $X_a$ and so of the operator $i\gamma^a \mathrm{Ad}_{X_a}$ in normal coordinates up to $o({\bf 1})$.

It is known that this spectral triple is finitely summable and regular. 
To compute the torsion functional for $\widetilde D$ we start 
(slightly abusing the notation) with its symbol expanded as:
\begin{equation} 
	\sigma(\widetilde D) = \gamma^a \bigl(-\xi_a  
	{ + \, i \,\mathrm{Ad}_{X_a} \bigr)   + o({\bf 1})}.
\end{equation}
	It is not difficult to see that the bimodule $\Omega^1_{\widetilde D}$ spanned by 
	$\alpha[\widetilde D,\beta ]$, $\alpha,\beta\in C^\infty(M,M_N(\mathbb{C}))$,
	is also spanned just by the elements $u=\gamma^a v_a$, 
	where we can expand $v_a\in M_N(\mathbb{C})+o({\bf 1})$.
	We now compute the symbols of $ \sigma(uvw \widetilde D) = 	\mathfrak{s}_1  + 	\mathfrak{s}_0$ for $u,v,w\in \Omega^1_{\widetilde D}$: 
	\begin{equation}
		\begin{aligned} 	 
			\mathfrak{s}_1 &= 
			-  \gamma^a \gamma^b \gamma^c \gamma^d u_av_bw_c \,\xi_d, \\	 
			\mathfrak{s}_0 &= 
			i\gamma^a \gamma^b \gamma^c \gamma^d u_av_bw_c \,\mathrm{ad}_{X_d}	.	 
		\end{aligned} \label{uvwD}
	\end{equation}	
	
	Next, to compute the two leading symbols  of $\widetilde D^{-2m}$,
	$\sigma(\widetilde D^{-2m}) =\mathfrak{c}_{2m} + \mathfrak{c}_{2m+1}$,
	we need the two leading  symbols of $\widetilde D^2$, 
	$\sigma(\widetilde D^{2}) = {\mathfrak a}_2 + {\mathfrak a}_1 + \ldots$,
	\begin{equation} \label{tildeD2T}
		\begin{aligned}
			{\mathfrak a}_2 =&  ||\xi||^2  + o({\bf x}), \\
			{\mathfrak a}_1 = & 
			{ -2 \mathrm{Ad}_{X_a}\xi_a  + o(\bf 1),		}
		\end{aligned}	
	\end{equation}
	and the two leading symbols of $\widetilde D^{-2}$,
	$\sigma(\widetilde D^{-2}) = {\mathfrak b}_2 + {\mathfrak b}_3 + \ldots$,
	\begin{equation} \label{tildeD-2T}
		\begin{aligned}
			{\mathfrak b}_2 =&  ||\xi||^{-2}  + o({\bf x}), \\
			{\mathfrak b}_3 = & 
			+ {2 \mathrm{Ad}_{X_a}\xi_a||\xi||^{-4}  + o(\bf 1).		}
		\end{aligned}	
	\end{equation}
	Then, using \cite[Lemma A.1]{DSZ23} we have:%
	\begin{equation} 
		\begin{aligned}
			&\mathfrak c_{2m} = ||\xi||^{-2m} + o(\bf x),\\
			&\mathfrak c_{2m+1}= 
			{ n \mathrm{Ad}_{X_a}\xi_a||\xi||^{-n-2}.}
		\end{aligned}\label{symD-2m}
	\end{equation}		
	Therefore, since $\partial_a \mathfrak c_{2m}(0)=0$,
	$$
	\sigma_{-n} \left(uvw\widetilde D|\widetilde D|^{-n}\right)=
	\mathfrak s_{0}\mathfrak c_{n}+\mathfrak s_{1}\mathfrak c_{n+1},
	$$
	where both terms on the right hand side contain linearly  $\mathrm{ad}_{X_a}$.
	It turns out however, that the trace of $\mathrm{ad}_{X}$, $X\in M_N(\mathbb{C})$ , considered
	as a linear operator on $M_N(\mathbb{C})$ vanishes. This could be checked explicitly, using the
	basis $E_{\alpha\beta}$, $\alpha,\beta=1,\ldots,N$ of $M_N(\mathbb{C})$, where the
	matrix  $E_{\alpha\beta}$  has the $\alpha,\beta$  matrix entry $1$ and otherwise $0$. 
	We compute:
	$$
	Ad_{E_{\mu\nu}}  E_{\alpha\beta} =
	E_{\mu\beta}\delta_{\alpha\nu}-E_{\alpha\nu}\delta_{\mu\beta}=
	K_{\mu\nu;\alpha\beta,\rho\tau} E_{\rho\tau},$$
	where $K_{\mu\nu}$ as a matrix is,
	$$K_{\mu\nu;\alpha\beta,\rho\tau}= 
	\delta_{\mu\rho}\delta_{\beta\tau}\delta_{\nu\alpha}
	- \delta_{\alpha\rho}\delta_{\nu\tau}\delta_{\mu\beta}.
	$$                                           ,	Thus,           ,                ,                ,
	$$ \hbox{Tr}\,  \mathrm{ad}_{E_{\mu\nu}} = K_{\mu\nu;\alpha\beta,\alpha\beta}=
	N(\delta_{\mu\alpha}\delta_{\nu\alpha}
	-\delta_{\nu\alpha}\delta_{\mu\alpha})=0.$$
	Therefore, the spectral torsion functional 
	${\mathcal T}(u,v,w)$ for $\widetilde{D}$ vanishes and so 
	Einstein-Yang-Mills spectral triple is torsion free.
\subsection{Almost commutative $\bf M \times \mathbb{Z}_2$}
We assume that $M$ is a closed spin manifold and $(C^\infty(M), D, \cH)$ is an even spectral triple  of  dimension $n$  with the standard Dirac operator $D$ and a 
grading $\chi$. We know that it is spectrally closed, so, in particular, $D$ is torsion-free. We consider the usual double-sheet spectral triple, 
		$(C^\infty(M)\otimes  \IC^2, {\mathcal D}, \cH \otimes \IC^2)$, 
with the Dirac operator,  
		$$  {\mathcal D} = \left(
		\begin{array}{cc} D & \chi \Phi \\ \chi \Phi^* & 
			D \end{array} \right), $$
where $\Phi \in \mathbb{C}$.             
 The bimodule of one forms, associated to ${\mathcal D}$ consists of the following operators,
		\begin{equation}
			\omega= \begin{pmatrix}
				w^+ &  \Phi \chi f^+ \\ \Phi^* \chi f^- & w^-
			\end{pmatrix},
		\end{equation}                       
		where $w^\pm \in \Omega^1_D(M)$ and $ f^\pm \in C^\infty(M)$.  We shall
		denote the diagonal part as $\omega^d$ and the off-diagonal part as $\omega^o$.
		
		In order to  calculate $\wres(\omega_1 \omega_2 \omega_3 \mathcal D \mathcal{D}^{-2m})$ for  one-forms  $\omega_1,\omega_2,\omega_3$
		let us observe that due to linearity it suffices to do it separately for diagonal and off-diagonal parts. Moreover,  the algebra rules in this extended Clifford algebra	are obvious and we restrict ourselves to the four possible cases.  Furthermore, since  $\mathcal{D}^{-2m}=D^{-2m}\mathbf 1(1-m |\Phi|^2D^{-2}+\dots)$ we  note that  in the Wodzicki residue formula $\mathcal{D}^{-2m}$ can be replaced by $D^{-2m}$. 
		
		By a straightforward calculation we get,                                                                                                                                                                                                                                                              
		\begin{equation}
			\begin{aligned}
				\wres \bigl( \omega_1^d \omega_2^d \omega_3^d \mathcal{D}	\mathcal{D}^{-2m} \bigr) & = 0, \\
				\wres \bigl( \omega_1^d \omega_2^d \omega_3^o	 \mathcal{D}\mathcal{D}^{-2m} \bigr) & = 
				\wres \bigl( |\Phi|^2 (w_1^+ w_2^+ f_3^+  + w_1^- w_2^- f_3^-) D^{-2m} \bigr) \\ 
				& = |\Phi|^2 \bigl( \mathcal{g}(w_1^+,w_2^+) f_3^+ +  \mathcal{g}(w_1^-,w_2^-) f_3^-  \bigr),\\
				\wres \bigl( \omega_1^d \omega_2^o \omega_3^o  \mathcal{D}	\mathcal{D}^{-2m} \bigr) & = 
				\wres \bigl( |\Phi|^2 (w_1^+ f_2^+ f_3^-  + w_1^- f_2^- f_3^+) D D^{-2m} \bigr) = 0, \\
				\wres \bigl( \omega_1^o \omega_2^o \omega_3^o	 \mathcal{D}\mathcal{D}^{-2m} \bigr) & = 
				\wres \bigl( |\Phi|^4(f_1^+ f_2^-  f_3^+  + f_1^- f_2^+ f_3^-)  D^{-2m} \bigr) \\
				& = |\Phi|^4 \mathcal{V}(f_1^+ f_2^-  f_3^+  + f_1^- f_2^+ f_3^-),
			\end{aligned}
		\end{equation}                                                      
{
where $\mathcal{g}(w_1,w_2)$ is proportional to 
$\int_M g(w_1,w_2)vol_g$
and	$\mathcal{V}(f)$ is proportional to $\int_M f vol_g$. 
}	
	Note that the first equation is a consequence of  $D$ being torsion-free,
	whereas the third equation holds if  $\wres(\omega D D^{-2m})=0$, which is the consequence
	of the spectral triple on $M$ being spectrally closed.
	                                     
As a consequence, we see that the product of the standard spectral triple over $M$ with
a finite discrete one does not satisfy the torsion-free condition unless $\Phi=0$. 
Note that the generalisation of the model with $\Phi$ being a complex-valued function
on $M$ does not change the conclusion.                                             
	
	\subsection{Conformally rescaled noncommutative tori}
	
	Consider the spectral triple on the noncommutative $n$-torus,
	$({\mathcal A }, {\mathcal H}, D_{k})$, where 
	$\mathcal A = C^\infty(\mathbb{T}^n_\theta)$ 
	is the	$\delta _{j}$-smooth subalgebra for the standard derivations $\delta_{j}$, $j=1,\dots,n$. Next, ${\mathcal H}= L^2(\mathbb{T}^n_\theta,\tau)$, where $\tau$ is the standard trace which annihilates $\delta _{j}$.
	Also, $D_{k}=kDk$ is the conformal rescaling of the standard (flat) Dirac operator $D=\sum _{j} \gamma ^{j} \delta _{j}$  with $\gamma ^{j}$ being the usual gamma matrices and the conformal 	factor $k>0$ is  from the copy ${\mathcal A}^{o}$ of 
	${\mathcal A}$  in the commutant of $\mathcal A$.  
	The bimodule of one-forms, generated by the commutators $[D_{k},{\mathcal a}]$, ${\mathcal a} \in {\mathcal A}$, is a free left module generated by $ k^{2} \gamma ^{j}$.  Using the calculus of pseudodifferential operators over noncommutative tori, following  \cite{CoMo14}, 	to $\hat{{\mathcal A}}$-valued symbols  (where $\hat{{\mathcal A}}$ is the algebra generated by ${\mathcal{A}}$
	and  ${\mathcal {A}^o}$) and of the analogue of the Wodzicki residue, we defined and computed various spectral functionals, as the metric and Einstein functionals \cite[Section 5.2.1]{DSZ23} in dimension 2 and 4. In particular, we demonstrated there that strictly irrational noncommutative 2- and 4-dimensional tori with a conformally rescaled Dirac operator are spectrally	closed. We extend here this result to any  dimension:
	\begin{thm}
		The spectral triple $\bigl(C^\infty(\mathbb{T}^n_\theta), {\mathcal H}, D_{k}\bigr)$
		for strictly irrational noncommutative torus of dimension $n$ is spectrally closed, that is for any $T$, operator of order $0$ (that is, a polynomial in 
		$a, [D_k,b]$, $a,b \in \mathcal{A}$)
		$$ \wres\bigl(T D_k |D_k|^{-n} \bigr)= 0.$$
	\end{thm}
	\begin{proof}
		The proof  is based on a simple observation that the symbol of order $-n$ of 
		$T D_k |D_k|^{-n}$ is a sum of terms that are of the form 
		$Y_{\alpha, j,\beta}(\xi) k^\alpha \delta_j(k) k^\beta$,  for $j=1,\ldots,n$,
		some integers  $\alpha,\beta$  and $C^\infty(\mathbb{T}^n_\theta)$-valued 
		functions $Y_{\alpha,j,\beta}(\xi)$, which 
		are homogeoneous in $\xi$ of degree $-n$.  Since for strictly irrational tori the trace
		over the full algebra $\hat{\mathcal A}$ factorizes,  it is sufficient
		to show that $\tau(k^\alpha \delta_j(k) k^\beta) =0$ for any $j,\alpha,\beta$.
		For that we write $k = e^{h}$ and observe that
		$$ \tau(k^\alpha \delta_j(k) k^\beta) = \tau\bigl(e^{(\alpha+\beta)h} \delta_j(e^{h})\bigr)$$ vanishes since for $\alpha+\beta \not= -1$
		it equals 
		$$\frac{1}{\alpha+\beta+1}\tau\bigl(\delta_j(e^{(\alpha+\beta+1)h})\bigr)$$ while for 
		$\alpha+\beta=-1$ it equals $\tau\bigl(\delta_j(h))\bigr)$.
		This ends the proof.
	$\qquad\Box$
	\end{proof}
	As an immediate consequence of spectral closedness, we have: 
\begin{coro}
	The Dirac operator $D_k$ over the $n$-dimensional strictly irrational noncommutative torus $\mathbb{T}^n_\theta$ is torsion free.
\end{coro}
	\subsection{Quantum  $SU(2)$. } 
	\ \\
	We briefly  recall the construction of the spectral triple $\bigl(\mathcal{A}(SU_q(2)),\mathcal{H},D\bigr)$
	of \cite{DLSSV04}. Let $\mathcal{A} = \mathcal{A}(SU_q(2))$ be the $*$-algebra generated	by $a$ and~$b$, subject to the  commutation rules:
	\begin{equation}
		\begin{aligned}
			ba = q ab,  \qquad  b^*a = qab^*, \qquad bb^* = b^*b, \\
			a^*a + q^2 b^*b = 1,  \qquad  aa^* + bb^* = 1,
			\label{eq:suq2-relns}
		\end{aligned}
	\end{equation}
	where $0 < q < 1$.  The 
	Hilbert space $\mathcal{H}$ has an orthonormal basis  $\ket{j\mu n\up}$ for $j = 0,\half,1,\dots$,
	$\mu = -j,\dots,j$ and $n = -j^+,\dots,j^+$; together with
	$\ket{j\mu n\dn}$ for $j=\half,1,\dots$, $\mu = -j,\ldots,j$ and
	$n = -j^-,\dots,j^-$, where $j^+ = j+ \half$ and $j^- = j -\half$. For the explicit formulas defining the representation 
 of $\mathcal{A}$ on $\mathcal{H}$ 
	we refer to \cite{DLSSV04} or \cite{DLSSV05}. The equivariant Dirac
	operator is chosen to be isospectral to the classical Dirac over $SU(2)$;
	\begin{equation}
		D \ket{j\mu n \up} = (2j + \sesq) \ket{j\mu n \up}, \qquad 
		D \ket{j\mu n \dn} = - (2j + \half)  \ket{j\mu n \dn}.
		\label{suq2-Dirac}
	\end{equation}
	In \cite{DLSSV05} it was shown that the triple is regular and its dimension spectrum of the spectral triple was determined, together with a an explicit formula that computes respective noncommutative residue. In particular,	the following theorem was proven.
	
	\begin{thm}
		If $T \in \Psi^0(\mathcal{A})$  is in the algebra of pseudodifferential operators of order $0$ (as defined in \cite{DLSSV05}) then for $P^\up, P^\dn$, projections on the respective up and down part of the Hilbert space of spinors, the noncommutative integral of $P^{(\up,\dn)}T |D|^{-2}$ can be explicitly computed (see \cite{DLSSV05}, Theorem 4.1) as:
		\begin{equation}
			\begin{aligned}                                                                                                                                                                                                                                                                                   
				\ncint {P^\dn }T |D|^{-2} &= \bigl(\tau_1 \ox \tau_0^\up +  \tau_0^\dn \ox \tau_1\bigr) \bigl(\beta(T)\bigr),
				\\
				\ncint {P^\up} T |D|^{-2} &= \bigl(\tau_1 \ox  \tau_0^\dn
				+\tau_0^\up  \ox \tau_1\bigr) 
{				\bigl(\beta(T)\bigr),	}
			\end{aligned}
			\label{residues}            
		\end{equation}
		where the noncommutative ingeral is, as usually,
		$$
		\ncint  T := \Res_{z=0} \mathrm{Tr} (T |D|^{-z}).
		$$
		The map 
{		$\beta$ }
is a projection on the two first legs of the zero-degree part of a 
		map $\rho$, which is a homomorphism 
		\begin{equation}
			\rho: \mathcal{B} \to \mathcal{A}(D^2_{q}) \ox \mathcal{A}(D^2_{q}) \ox \mathcal{A}(\mathbb{S}^1),
			\label{eq:symbol-map-bis}
		\end{equation}
		where $\mathcal{B}$ is a subalgebra of  $\Psi^0(\mathcal{A})$ generated by diagonal
		parts of the generators $a,b$ and their commutators with $D$ and  $|D|$,   
		and $ \mathcal{A}(D^2_{q})$ are quantum discs. Using the symbol map $\sigma: 
		\mathcal{A}(D^2_{q}) \to \mathcal{A}(\mathbb{S}^1)$ and the standard irreducible 
		representation of quantum discs $\pi$ the functionals $\tau_1, \tau_0^\up, \tau_0^\dn$
		are defined for $x \in \mathcal{A}(D^2_{q})$, as
		\begin{equation}
			\begin{aligned}
				\tau_1(x) &:= \frac{1}{2\pi} \int_{S^1} \sigma(x),
				\\
				\tau_0^\up(x) &:= \lim_{N\to\infty} \Tr_N \pi(x) - (N+\sesq) \tau_1(x),
				\\
				\tau_0^\dn(x) &:= \lim_{N\to\infty} \Tr_N \pi(x) - (N+\half) \tau_1(x).
			\end{aligned}
		\end{equation}
	\end{thm}
	
Using this  theorem we can state as a simple corollary,
	\begin{coro}
		The standard equivariant spectral triple over $\mathcal{A}(SU_q(2))$ is {\em spectrally
			closed} and, in particular, the Dirac operator \eqref{suq2-Dirac} is {\em torsion-free}.
	\end{coro}
	\begin{proof}
		We compute, {using $D = (P^\up - P^\dn) |D|$},
	$$ 	\ncint  T D |D|^{-3} = \ncint  T  (P^\up - P^\dn) 
{	|D|^{-2}  }
		=\bigl(\tau_1 \ox  ( \tau_0^\dn- \tau_0^\up ) +  (\tau_0^\up-\tau_0^\dn )\ox \tau_1\bigr) \bigl(\beta(T)\bigr),$$
		where we used \eqref{residues}.                              dn         up    	Then, from the definition of traces we see that:
		$$ \tau_0^\dn(x) - \tau_0^\up(x) = \tau_1(x), $$
		and therefore                      
			$$ 	\ncint  T D |D|^{-3} = \bigl(\tau_1 \ox  (\tau_1) +  (-\tau_1)\ox \tau_1\bigr) \bigl(\beta(T)\bigr) = 0. 	\qquad\qquad\Box$$ 
                                                
	\end{proof}
	
	\section{Final remarks}
We introduced a tangible notion of torsion for finite summable, regular spectral triples 
with a generalised noncommutative trace via a spectral trilinear functional of Dirac one-forms. 
It allows a direct verification whether the spectral triple (or the Dirac operator) is torsion-free.  Four of the discussed examples are indeed torsion-free except the simplest case of an almost-commutative geometry. We conjecture that the spectral triples over almost-commutative geometries with non-simple internal algebra, nontrivially coupled by $D$, have a non-vanishing torsion. We hope that in many other cases, the spectral torsion functional will enable a more advanced study of the families of Dirac operators.

	\section*{Declarations}
	LD acknowledges that this work is part of the project Graph Algebras partially supported by EU grant HORIZON-MSCA-SE-2021 Project 101086394 and
	was partially supported the University of Warsaw Thematic Research Programme "Quantum Symmetries".
	
	AS and PZ acknowledge that this work is supported by the Polish National Science Centre grant 2020/37/B/ST1/01540.

	\pagebreak
	

\begin{thebibliography}{xx}
		
		\bibitem{BM20} E.J.~Beggs  and S.~Majid,
		Quantum Riemannian Geometry, Springer (2020).
		
		\bibitem{BGJ21a} 
		J.~Bhowmick, D.~Goswami and S.~Joardar,
		{\it Levi-Civita connections for conformally deformed metrics on tame differential calculi}, International Journal of Mathematics, 32, (2021) 2150088.
		
		\bibitem{BGJ21b} 
		J.~Bhowmick, D.~Goswami, and S.~Joardar,
		{\it A new look at Levi-Civita connection in noncommutative geometry}
		International Journal of Geometric Methods in Modern Physics 18 (2021) 215010.
		
		\bibitem{Ca23} É. Cartan, 
			{\it Sur les variétés à connexion affine et la théorie de la rélativité généralisée}, Ann. Éc. Norm. Sup. 40 (1923) 325–412, Ann. Éc. Norm. Sup. 41 (1924) 1–25, Ann. Éc. Norm. Sup. 42 (1925) 17–88.
		
		\bibitem{Co80} A.~Connes, 	{\it C* alg\`ebres et g\'eom\'etrie diff\'erentielle}, C.\,R.\,Acad.\,Sci.\,Paris S\'er.\,A-B 290 (1980) A599--A604.
		
		\bibitem{Co94} A.~Connes, {Noncommutative Geometry},
		Academic Press, Cambridge (1994).
		
	
		\bibitem{CoMo95}	A.~Connes and H.~Moscovici, 	{\it The local index formula in noncommutative geometry},  	Geom. Funct. Anal. 5 (1995) 174-243.
		
		\bibitem{CoMo14}
		A.~Connes and H.~Moscovici, 	{\it Modular curvature for noncommutative two-tori},	J. Am. Math. Soc. 27(3) (2014) 639-684.
		
		
		\bibitem{DLSSV04}
		L.~D\k{a}browski, G.~ Landi, A.~ Sitarz, W. ~van Suijlekom, and J.C.~ Varilly,  	{\it The Dirac operator on SUq(2)}, Comm. Math. Phys. 259 (2005), 729-759.
		
		\bibitem{DLSSV05}
		L.~D\k{a}browski, G.~ Landi, A.~ Sitarz, W. ~van Suijlekom, and J.C.~ Varilly, 	{\it The local index formula for SUq(2)}, K-Theory, 35, 3, (2005) 375-394. 
		
	
		\bibitem{DSZ23}
		L.~D\k{a}browski, A.~Sitarz, and P.~Zalecki, 
		{\it Spectral Metric and Einstein Functionals}, Adv. Math. Vol. 427, (2023) 1091286. 
		
		\bibitem{DSZ23a}
		L.~D\k{a}browski, A.~Sitarz, and  P.~Zalecki, 
		{\it Spectral Metric and Einstein Functionals for Hodge-Dirac operator},
	    {\tt arXiv:2307.14877}
				
		
		\bibitem{FrSu78}
		T.~Friedrich  and  S.~Sulanke, 	{\it Ein Kriterium f\"ur die formale Selbstadjungiertheit des Dirac-Operators}, Colloq. Math. 40 (2) (1978/1979) 239–247.
		
		
		\bibitem{Gu85}
		V. Guillemin, 	{\it A new proof of Weyl's formula on the asymptotic distribution of eigenvalues}, 
		Adv. Math. 55(2) (1985) 131--160. 
		
		
		\bibitem{HPS10}
		F.~Hanisch, F.~Pfäffle,  and C.A.~Stephan, 	{\it The spectral action for Dirac operators with skew-symmetric torsion}, Comm. Math. Phys. 300 (3) (2010) 877–888. 
		
		\bibitem{HHKN76}
		F.W. Hehl, P. von der Heyde, G.D. Kerlick,  and J.N. Nester, 	{\it General relativity with spin and torsion: Foundations and prospects}, Rev. Modern Phys. 48 (1976) 393–416.
		
		\bibitem{Hi06} N.~Higson,
	    	{\it The residue index theorem of Connes and Moscovici}, Surveys in noncom-
		mutative geometry, Clay Math. Proc., vol. 6, Amer. Math. Soc., Providence, RI, 
		2006, pp. 71–126.
		
		
		\bibitem{ILV08}
		B.~Iochum, C.~Levy,  and D.~Vassilevich, 
			{\it Spectral action for torsion with and without boundaries}, 
		Comm. Math. Phys. 310 (2012) 367-382.
		
		\bibitem{MRvS22} {
		B.~Mesland, A.~Rennie, W.~van Suijlekom,
		{\it Curvature of differentiable Hilbert modules and Kasparov modules},
		Advances in Mathematics, 402, (2022), 108128.}
		
		\bibitem{MSV99}
		U.~ Mueller, Ch.~Schubert,  and  A.~van de Ven,
			{\it A Closed Formula for the Riemann Normal Coordinate Expansion},
		Gen. Rel. Grav. 31 (1999), 1759-1768.
		
		\bibitem{PfSt12}
		F.~Pf\"affle  and C.A.~Stephan, 
			{\it On gravity, torsion and the spectral action principle}, 
		J. Func. Anal. 262 (2012) 1529–1565.
		
		\bibitem{Sh02}
		I.L.~Shapiro, 	{\it Physical aspects of the space–time torsion}, Phys. Rep. 357 (2002) 113–213.
		
		\bibitem{Si14}
		A.~Sitarz, 
			{\it Wodzicki residue and minimal operators on a noncommutative 4-dimensional torus}, 
		J. Pseudo-Differ. Oper. Appl., 3  (2014) 305-317.	
		
		\bibitem{SS12} {
		S.~Sternberg, {\it Curvature in Mathematics and Physics},
		Dover books on mathematics,  Courier Corporation, 2012} 
		
		
		\bibitem{Ta81}
		M.~Taylor, 	Pseudodifferential operators, 
		Princeton University Press (1981)

		\bibitem{Uuye11}
			O.~Uuye, 	{\it Pseudo-differential operators and regularity of spectral triples}, Fields Communications Series {61}, (2011)
			Perspectives on noncommutative geometry, 153--163.
		
		\bibitem{Wo87}
		M.~Wodzicki, 	{\it Noncommutative residue. I. Fundamentals}, 
		in K-Theory, Arithmetic and Geometry (Moscow, 1984-1986), 
		Lecture Notes in Mathematics Vol. 1289 (Springer, Berlin, 1987), 320--399.
		
	\end{thebibliography}
\end{document}